\documentclass[11pt, twoside, fullpage]{article}
\usepackage{amssymb, amsmath, amsthm}
\usepackage[top=30mm, left=23mm, right=23mm, bottom=30mm]{geometry}
\usepackage{inputenc}
\usepackage{fancyhdr}
\usepackage{tikz}
\usepackage{verbatim}
\usetikzlibrary{positioning}
\usepackage{titling}
\usepackage{url}
\usepackage{hyperref}
\usepackage[T1]{fontenc}
\usepackage{currvita}
\usepackage{xcolor}
\predate{}
\postdate{}
\usepackage{lipsum}
\newtheorem{theorem}{Theorem}
\newtheorem{lemma}[theorem]{Lemma}

\newtheorem{question}[theorem]{Question}

\newcommand{\F}{\mathcal{F}}

\newcommand{\pn}{\mathcal{P}([n])}
\begin{document}
\pagestyle{fancy}
\fancyhf{}
\fancyhead [LE, RO] {\thepage}
\fancyhead [CE] {DAVID ELLIS, MARIA-ROMINA IVAN AND IMRE LEADER}
\fancyhead [CO] {SMALL SETS IN UNION-CLOSED FAMILIES}
\renewcommand{\headrulewidth}{0pt}
\renewcommand{\l}{\rule{6em}{1pt}\ }
\title{\Large{\textbf{SMALL SETS IN UNION-CLOSED FAMILIES}}}
\author{DAVID ELLIS, MARIA-ROMINA IVAN AND IMRE LEADER}
\date{}

\maketitle
\begin{abstract}
Our aim in this note is to show that, for any $\epsilon>0$, there exists a union-closed family $\mathcal F$ with (unique) smallest set $S$ such that no element of $S$ belongs to more than a fraction $\epsilon$ of the sets in $\mathcal F$. More precisely, we give an example of a union-closed family with smallest set of size $k$ such that no element of this set belongs to more than a fraction $(1+o(1))\frac{\log_2 k}{2k}$ of the sets in $\mathcal F$. \par We also give explicit examples of union-closed families containing `small' sets for which we have been unable to verify the Union-Closed Conjecture.
\end{abstract}
\section{Introduction}

If $X$ is a set, a family $\mathcal{F}$ of subsets of $X$ is said to be {\em union-closed} if the union of any two sets in $\F$ is also in $\F$. The Union-Closed Conjecture (a conjecture of Frankl \cite{duffus}) states that if $X$ is a finite set and $\F$ is a union-closed family of subsets of $X$ (with $\F \neq \{\emptyset\}$), then there exists an element $x \in X$ such that $x$ is contained in at least half of the sets in $\F$. Despite the efforts of many researchers over the last forty-five years, and a recent Polymath project \cite{polymath} aimed at resolving it, this conjecture remains wide open. It has only been proved under very strong constraints on the ground-set $X$ or the family $\F$; for example, Balla, Bollob\'as and Eccles \cite{bbe} proved it in the case where $|\F| \geq \tfrac{2}{3} 2^{|X|}$; more recently, Karpas \cite{karpas} proved it in the case where $|\F| \geq (\tfrac{1}{2}-c)2^{|X|}$ for a small absolute constant $c >0$; and it is also known to hold whenever $|X| \leq 12$ or $|\F| \leq 50$, from work of Vu\v{c}kovi\'c and \v{Z}ivkovi\'c \cite{vz} and of Roberts and Simpson \cite{rs}. Note that the Union-Closed Conjecture is not even known to
hold in the weaker form where we replace the fraction $1/2$ by any other
fixed $\epsilon>0$.\footnote{{\em Note added in proof:} shortly before the acceptance of this manuscript, Gilmer [arXiv:2211.09055] obtained a breakthrough on the Union-Closed Conjecture, showing that it holds in the weaker form with the fraction
$1/2$ replaced by $1/100$.} For general background and a wealth of further information on the Union-Closed Conjecture see the survey of Bruhn and Schaudt \cite{bsbook}.

As usual, if $X$ is a set we write $\mathcal{P}(X)$ for its power-set. If $X$ is a finite set and $\F \subset \mathcal{P}(X)$ with $\F \neq \emptyset$, we define the {\em frequency} of $x$ (with respect to $\F$) to be $\gamma_x = |\{A \in \F:\ x \in A\}|/|\F|$, i.e., $\gamma_x$ is the proportion of members of $X$ that contain $x$. If a union-closed family contains a `small' set, what can we say about the frequencies in that set?\par If a union-closed family $\mathcal F$ contains a singleton, then that element clearly has frequency at least $1/2$, while if it contains a set $S$ of size 2 then, as noted by Sarvate and Renaud \cite{sr}, some element of $S$ has frequency at least $1/2$. However, they also gave an example of a union-closed family $\mathcal F$ whose smallest set $S$ has size 3 and yet where each element of $S$ has frequency below $1/2$. Generalising a construction of Poonen \cite{poonen}, Bruhn and Schaudt \cite{bsbook} gave, for each $k\geq 3$, an example of a union-closed family with (unique) smallest set of size $k$ and with every element of that set having frequency below $1/2$.\par However, in these and all other known examples, there is always some element of a minimal-size set having frequency at least $1/3$. So it is natural to ask if there is really a constant lower bound for these frequencies.\par Our aim in this note is to show that this is not the case.
\begin{theorem}
\label{thm:main}
For any positive integer $k$, there exists a union-closed family in which the (unique) smallest set has size $k$, but where each element of this set has frequency
$$(1+o(1))\frac{\log k}{2k}.$$
\end{theorem}

\noindent (All logarithms in this paper are to base 2. Also, as usual, the $o(1)$ denotes a function of $k$ that tends to zero as $k$ tends to infinity.) The proof of Theorem \ref{thm:main} is by an explicit construction.\\
\par Theorem \ref{thm:main} is asymptotically sharp, in view of results of W\'ojcik \cite{wojcik} and Balla \cite{balla}: W\'ojcik showed that if $S$ is a set of size $k \geq 1$ in a finite union-closed family, then the average frequency of the elements in $S$ is at least $c_k$, where $k\cdot c_k$ is defined to be the minimum average set-size over all union-closed families on the ground-set $[k]$, and Balla showed that $c_k = (1+o(1))\frac{\log k}{2k}$, confirming a conjecture of W\'ojcik from \cite{wojcik}. \\

\par Remarkably, there are union-closed families containing small sets, even sets of size 3, for which we have been unable to verify the Union-Closed Conjecture. We give some examples at the end of the paper.

\section{Proof of main result}

For our construction, we need the following `design-theoretic' lemma.
\begin{lemma}
For any positive integers $k>t$ there exist infinitely many positive integers $d$ such that $t$ divides $dk$ and the following holds. If $X$ is a set of size $dk/t$, then there exists a family $\mathcal{A} = \{A_1,\ldots,A_k\}$ of $k$ $d$-element subsets of $X$, such that each element of $X$ is contained in exactly $t$ sets in $\mathcal{A}$, and for $2 \leq r \leq t$, any $r$ distinct sets in $\mathcal{A}$ have intersection of size
$$d\dfrac{(t-1)(t-2)\ldots(t-r+1)}{(k-1)(k-2)\ldots(k-r+1)},$$
i.e.\
$$|A_{i_1}\cap A_{i_2}\cap\ldots\cap A_{i_r}|=d\dfrac{(t-1)(t-2)\ldots(t-r+1)}{(k-1)(k-2)\ldots(k-r+1)}$$ for any $1\leq i_1 < i_2< \ldots <i_r \leq k$.
\label{structure}
\end{lemma}
\begin{proof}
Let $q$ be a positive integer, and set $d=\binom{k-1}{t-1} q^t$; we will take $|X|=\binom{k}{t}q^t$. Partition $[qk]$ into $k$ sets, $B_1, B_2,\ldots, B_k$ say, each of size $q$; we call these sets `blocks'. We let $X$ be the set of all $t$-element subsets of $[qk]$ that contain at most one element from each block. For each $i \in [k]$ we let $A_i$ be the family of all sets in $X$ that contain an element from the block $B_i$. Clearly, $|A_i|=\binom{k-1}{t-1}q^t=d$ for each $i \in [k]$, and each element of $X$ appears in exactly $t$ of the $A_i$. Also, for example $A_i\cap A_j$ consists of all sets in $X$ that contain both an element from the block $B_i$ and an element from the block $B_j$, so
$$|A_i \cap A_j| = \binom{k-2}{t-2}q^t=\binom{k-1}{t-1}q^t\dfrac{t-1}{k-1}=d\dfrac{t-1}{k-1}.$$
It is easy to check that the other intersections also have the claimed sizes.
\end{proof}

We remark that, in what follows, it is vital that the integer $d$ in Lemma \ref{structure} can be taken to be arbitrarily large as a function of $k$ and $t$.\\

\begin{proof}[Proof of Theorem \ref{thm:main}]
We define $n = dk/t+k$, we take $d \in \mathbb{N}$ as in the above lemma, and we let $X = [dk/t]$; the claim yields a family $\mathcal{A} = \{A_1,\ldots,A_k\}$ of $k$ $d$-element subsets of $X = [dk/t]$ such that each element of $[dk/t]$ is contained in exactly $t$ of the sets in $\mathcal{A}$, and for any $2 \leq r \leq t$, any $r$ distinct sets in $\mathcal{A}$ have intersection of size
$$d\dfrac{(t-1)(t-2)\ldots(t-r+1)}{(k-1)(k-2)\ldots(k-r+1)}.$$
Write $m= dk/t$. We take $\F \subset \pn$ to be the smallest union-closed family containing the $k$-element set $\{m+1,\ldots,m+k\}$ and all sets of the form $\{m+i\} \cup (X \setminus \{x\})$ where $i \in [k]$ and $x \in A_i$.

For brevity, we write $S_0 = \{m+1,m+2,\ldots,m+k\}$. We will show that each element of $S_0$ has frequency
$$(1+o(1)) \frac{\log k}{2k},$$
provided $t$ and $d$ are chosen to be appropriate functions of $k$; moreover, with these choices, $S_0$ will be the smallest set in $\F$.

Clearly, $\F$ contains $S_0$, all sets of the form $S_0 \cup (X \setminus \{x\})$ for $x \in X$, all sets of the form $R \cup X$ where $R$ is a nonempty subset of $S_0$, and finally all sets of the form $R \cup (X \setminus \{x\})$, where $R = \{m+i_1,\ldots,m+i_r\}$ is a nonempty $r$-element subset of $S_0$ and $x \in A_{i_1} \cap A_{i_2} \cap \ldots \cap A_{i_r}$, for $1 \leq r \leq t$. It is easy to see that the family $\mathcal{F}$ contains no other sets.

It follows that
\begin{align*}|\mathcal{F}| & = 1+dk/t+(2^k-1)+\sum_{r=1}^{t}{k \choose r} d\dfrac{(t-1)(t-2)\ldots(t-r+1)}{(k-1)(k-2)\ldots(k-r+1)}\\
&= dk/t+2^k+\frac{dk}{t}\sum_{r=1}^{t} {t \choose r}\\
&= dk/t+2^k+\frac{dk}{t}(2^t-1)\\
& = 2^k + \frac{dk2^t}{t}.
\end{align*}

On the other hand, the number of sets in $\mathcal{F}$ that contain the element $m+1$ is equal to
\begin{align*} & 1+dk/t+2^{k-1}+\sum_{r=1}^{t}{k-1 \choose r-1} d\dfrac{(t-1)(t-2)\ldots(t-r+1)}{(k-1)(k-2)\ldots(k-r+1)}\\
& = 1+dk/t+2^{k-1}+d\sum_{r=1}^{t}{t-1 \choose r-1}\\
& = 1+dk/t+2^{k-1}+2^{t-1}d.
\end{align*}

It follows that the frequency of $m+1$ (or, by symmetry, of any other element of $S_0$) equals
$$\frac{1+kd/t+2^{k-1}+2^{t-1}d}{2^k+dk2^t/t} = \frac{(1+2^{k-1})/d+k/t+2^{t-1}}{2^k/d+k2^t/t}.$$

To (asymptotically) minimise this expression, we take $t = \lfloor \log k \rfloor$ and $d \to \infty$ (for fixed $k$); this yields a union-closed family in which the (unique) smallest set (namely $S_0$) has size $k$, and every element of that set has frequency
$$(1+o(1))\frac{\log k}{2k},$$
proving the theorem.
\end{proof}

\section{An open problem}

We now turn to some explicit examples of union-closed families containing small sets for which we have been unable to establish the Union-Closed Conjecture. For simplicity, we concentrate on the most striking case, when the family contains a set of size 3, and indeed is generated by sets of size 3.\\

\par Our families live on ground-set $\mathbb Z_n^2$, the $n\times n$ torus. 
\begin{question}
\label{conj:tiles-in-grid}
Let $n \in \mathbb{N}$ and let $R \subset \mathbb{Z}_n$ with $|R|=3$. Does the Union-Closed Conjecture hold for the union-closed family $\mathcal F$ of subsets of $\mathbb{Z}_n^2$ generated by all the translates of $R \times \{0\}$ and of $\{0\} \times R$?
\end{question}
\par (Here, as usual, we say a union-closed family $\mathcal{F}$ is {\em generated by} a family $\mathcal{G}$ if it consists of all unions of sets in $\mathcal G$.)

Perhaps the most interesting case is when $n$ is prime. In that case we may assume that $R=\{0,1,r\}$ for some $r$, and so one feels that the verification of the Union-Closed Conjecture should be a triviality, but it seems not to be. Note that all the families in Question \ref{conj:tiles-in-grid} are transitive families, in the sense that all points `look the same', so that the Union-Closed Conjecture is equivalent to the assertion that the average size of the sets in the family is at least $n^2/2$.\par We mention that the corresponding result in $\mathbb Z_n$ (in other words, the union-closed family on ground-set $\mathbb Z_n$ generated by translates of $R$) is known to hold: this is proved in \cite{ael}.\\

\par We have verified the special case of Question \ref{conj:tiles-in-grid} where $R = \{0,1,2\}$. A sketch of the proof is as follows. Assume that $n \geq 6$, and let $\mathcal{F} \subset \mathcal{P}(\mathbb{Z}_n^2)$ be the union-closed family generated by all translates of $\{0,1,2\} \times \{0\}$ and of $\{0\} \times \{0,1,2\}$ (we call these translates {\em 3-tiles}, for brevity). Let $C = \{0,1,2,3\}^2$, a $4\times 4$ square. Consider the bipartite graph $H = (\mathcal{X},\mathcal{Y})$ with vertex-classes $\mathcal{X}$ and $\mathcal{Y}$, where $\mathcal{X}$ consists of all subsets of $C$ with size less than 8 that are intersections with $C$ of sets in $\F$, $\mathcal{Y}$ consists of all subsets of $C$ with size greater than 8 that are intersections with $C$ of sets in $\F$, and we join $S \in \mathcal{X}$ to $S' \in \mathcal{Y}$ if $|S'| + |S| \geq 16$ and $S' =S \cup U $ for some union $U$ of 3-tiles that are contained within $C$. It can be verified (by computer) that $H$ has a matching $m:\mathcal{X} \to \mathcal{Y}$ of size $|\mathcal{X}|=16520$. Such a matching $m$ gives rise to an injection
$$f:\{S \in \F:\ |S \cap C| < |C|/2\} \to \{S \in \F:\ |S \cap C| > |C|/2\}$$ given by $$\quad f(S) = (S \setminus C) \cup m(S \cap C)$$
with the property that $|S \cap C|+|f(S)\cap C| \geq |C|$ for all $S \in \F$ with $|S \cap C| < |C|/2$. It follows that a uniformly random subset of $\F$ has intersection with $|C|$ of expected size at least $|C|/2$, which in turn implies that there is an element of $C$
 with frequency at least $1/2$ (and in fact, since $\F$ is transitive, every element has frequency at least 1/2).
 
 We remark that this proof does not work if one tries to replace $C = \{0,1,2,3\}^2$ by $\{0,1,2\}^2$, as the resulting bipartite graph $H' = (\mathcal{X}',\mathcal{Y}')$ does not contain a matching of size $|\mathcal{X}'|$.

We remark also that it would be nice to find a non-computer proof of the above result.\\

\par {\bf Acknowledgement}: We are very grateful to Igor Balla for bringing the papers \cite{balla} and \cite{wojcik} to our attention.

\end{document}